\documentclass[11pt]{article}

\usepackage{amsmath}
\usepackage{amsfonts}
\usepackage{color}
\usepackage{amssymb}
\usepackage{graphicx}
\usepackage{hyperref}
\usepackage{enumerate}
\usepackage[all]{xy}

\def\tr{{\rm tr}}

 \allowdisplaybreaks

\begin{document}

\newtheorem{problem}{Problem}

\newtheorem{theorem}{Theorem}[section]
\newtheorem{corollary}[theorem]{Corollary}
\newtheorem{definition}[theorem]{Definition}
\newtheorem{conjecture}[theorem]{Conjecture}
\newtheorem{question}[theorem]{Question}
\newtheorem{lemma}[theorem]{Lemma}
\newtheorem{proposition}[theorem]{Proposition}
\newtheorem{quest}[theorem]{Question}
\newtheorem{example}[theorem]{Example}

\newenvironment{proof}{\noindent {\bf
Proof.}}{\rule{2mm}{2mm}\par\medskip}

\newenvironment{proofof3}{\noindent {\bf
Proof of  Theorem 1.2.}}{\rule{2mm}{2mm}\par\medskip}

\newenvironment{proofof5}{\noindent {\bf
Proof of  Theorem 1.3.}}{\rule{2mm}{2mm}\par\medskip}

\newcommand{\remark}{\medskip\par\noindent {\bf Remark.~~}}
\newcommand{\pp}{{\it p.}}
\newcommand{\de}{\em}

\title{  {Another  determinantal inequality involving partial traces }
 \thanks{This work was supported by  NSFC (Grant Nos. 11671402, 11871479),  
Hunan Provincial Natural Science Foundation (Grant Nos. 2016JJ2138, 2018JJ2479) 
and  Mathematics and Interdisciplinary Sciences Project of CSU. 
 E-mail addresses: ytli0921@hnu.edu.cn(Y. Li), 
fenglh@163.com (L. Feng), 
wjliu6210@126.com (W. Liu), 
FairyHuang@csu.edu.cn (Y. Huang, corresponding author). }}

\author{Yongtao Li$^a$, Lihua Feng$^b$, Weijun Liu$^b$, Yang Huang$^{\dag, b}$\\
{\small ${}^a$School of Mathematics, Hunan University} \\
{\small Changsha, Hunan, 410082, P.R. China } \\
{\small $^b$School of Mathematics and Statistics, Central South University} \\
{\small New Campus, Changsha, Hunan, 410083, P.R. China. } }

\maketitle

\vspace{-0.5cm}

\begin{abstract}
Let $A$ be a positive semidefinite $m\times m$ block matrix with each block $n$-square, 
then the following determinantal inequality for partial traces holds 
\[ 	(\tr A)^{mn} - \det(\tr_2 A)^n \ge 
  \bigl|   \det A - \det(\tr_1 A)^m \bigr|, \]
where $\tr_1$ and $\tr_2$ stand for the first and second partial trace, respectively. 
This result improves a recent result of Lin \cite{Lin16}.  
 \end{abstract}

{2010 Mathematics Subject Classication.  15A45, 15A60, 47B65.}

\section{Introduction}

\label{sec1} 

Throughout the paper, we use the following standard notation. 
The set of $n\times n$ complex matrices is denoted by $\mathbb{M}_n(\mathbb{C})$, 
or simply by $\mathbb{M}_n$, 
and the identity matrix of order $n$ by  $I_n$, or $I$ for short. 
If $A=[a_{ij}]$ is of order $m\times n$ and  
 $B$ is of order $s\times t$, the tensor product of $A$ with $B$, 
denoted by $A\otimes B$, is an $ms\times nt$ matrix, 
partitioned into $m\times n$ block matrix with the $(i,j)$-block the $s\times t$ matrix $a_{ij}B$.
In this paper, 
we are interested in complex block matrices. Let $\mathbb{M}_m(\mathbb{M}_n)$ 
be the set of complex matrices partitioned into $m\times m$ blocks 
with each block being a $n\times n$ matrix. 
The element of $\mathbb{M}_m(\mathbb{M}_n)$ is usually written as ${ A}=[A_{i,j}]_{i,j=1}^m$, 
where $A_{i,j}\in \mathbb{M}_n$ for all $i,j$. 
By convention, if $X\in \mathbb{M}_n$ is positive semidefinite, we write $X\ge 0$. 
For two Hermitian matrices $A$ and $B$ of the same size, $A\ge B$ means $A-B\ge 0$.  
It is easy to see that $\ge$ is a partial ordering on the set of Hermitian matrices, 
referred to as {\it L\"{o}uner ordering}.

Now we introduce the definition of partial traces, 
which comes from Quantum Information Theory \cite[p. 12]{Petz08}.  
For $A\in \mathbb{M}_m(\mathbb{M}_n)$, 
the first partial trace (map) $A \mapsto \mathrm{tr}_1 A \in \mathbb{M}_n$ is defined as the  
adjoint map of the imbedding map $X \mapsto I_m\otimes X\in \mathbb{M}_m\otimes \mathbb{M}_n$. 
Correspondingly, the second partial trace (map)  $A \mapsto \mathrm{tr}_2 A\in \mathbb{M}_m$ is 
defined as the adjoint map of the imbedding map 
$Y\mapsto Y\otimes I_n \in \mathbb{M}_m\otimes \mathbb{M}_n$. Therefore, we have
\begin{equation*} \label{eqdef} 
\langle I_m\otimes X, A \rangle =\langle X, \mathrm{tr}_1A \rangle ,
\quad \forall X\in \mathbb{M}_n, 
\end{equation*}
and 
\[ \langle Y\otimes I_n, A \rangle =\langle Y,\mathrm{tr}_2 A \rangle, 
\quad \forall Y\in \mathbb{M}_m. \]
Assume that $A=[A_{i,j}]_{i,j=1}^m$ with $A_{i,j}\in \mathbb{M}_n$, 
then the visualized forms of the partial traces 
are actually given in  \cite[pp. 120--123]{Bh07} as
\begin{equation*} \label{eqdef2}
 \mathrm{tr}_1 { A}=\sum\limits_{i=1}^m A_{i,i},\quad 
\mathrm{tr}_2{ A}=\bigl[ \mathrm{tr}A_{i,j}\bigr]_{i,j=1}^m. 
\end{equation*}

If ${ A}=[A_{i,j}]_{i,j=1}^m \in \mathbb{M}_m(\mathbb{M}_n)$ is positive semidefinite,  
it is easy to see that  both $\tr_1 A$ and $\tr_2 A$ 
are positive semidefinite; see, e.g., \cite{Zha12}. 
To some extent, 
these two partial traces are closely related.  
For instance, Audenaert \cite{Aud07} proved an inequality for Schatten $p$-norms,  
\begin{equation} \label{eqaud}
\tr A + \lVert A \lVert_q \ge \lVert \tr_1 A \rVert_q + \lVert \tr_2 A \rVert_q .
\end{equation}
Inequality (\ref{eqaud}) was used to prove the subadditivity of Tsallis entropies. 

Moreover, Ando (see \cite{Ando14}) established  the following,  

\begin{equation} \label{eqando}
(\tr A)I_m\otimes I_n+  A \ge  
I_m\otimes (\mathrm{tr}_1 A)  + (\mathrm{tr}_2 A) \otimes I_n,
\end{equation}
where $\ge $ means the L\"{o}uner ordering. 
Furthermore, 
Motivated by inequalities (\ref{eqaud}) and (\ref{eqando}), 
Lin \cite{Lin16} proved an analogous result for determinant, which states that 
\begin{equation} \label{eqlin}
(\tr A)^{mn} +\det A   \ge   \det (\tr_1 A)^m +\det (\tr_2 A)^n. 
\end{equation}

In this paper, we improve Lin's result (\ref{eqlin}) as follows. 

\begin{theorem} \label{thm34}
 Let $A\in \mathbb{M}_m(\mathbb{M}_n)$ be positive semidefinite. Then
\[ 	(\tr A)^{mn} - \det(\tr_2 A)^n \ge 
  \bigl|  \det A - \det(\tr_1 A)^m \bigr|. \]
\end{theorem} 

The paper is organized as follows. 
We first present some auxiliary results, 
and then we show our proof of Theorem \ref{thm34}. 
Finally, we extend our result to a larger class of matrices, namely, 
matrices whose numerical ranges are contained in a sector (Theorem \ref{thm36}).

\section{Auxiliary results and proofs}
\label{sec2}

For $A=[A_{i,j}]_{i,j=1}^m\in \mathbb{M}_m(\mathbb{M}_n)$, 
we define the partial tranpose of $A$ by $A^{\tau}=[A_{j,i}]_{i,j=1}^m$. 
It is clear that $A\ge 0$ does not necessarily imply $A^{\tau}\ge 0$. 
If both $A$ and $A^{\tau}$ are positive semidefinite, 
then $A$ is called to be {\it positive partial tranpose} (or PPT for short). 
Recall that a linear map $\Phi: \mathbb{M}_n\to \mathbb{M}_k$ is called positive if it maps positive matrices 
to positive matrices. 
A linear map $\Phi: \mathbb{M}_n\to \mathbb{M}_k$  is said to be $m$-positive if 
for $[A_{i,j}]_{i,j=1}^m\in \mathbb{M}_m(\mathbb{M}_n)$, 
\begin{equation}  \label{eq1}
[A_{i,j}]_{i,j=1}^m \ge 0 \Rightarrow [\Phi (A_{i,j})]_{i,j=1}^m\ge 0. 
\end{equation}
It is said to be {\it completely positive} 
if (\ref{eq1}) holds for any integer $m\ge 1$. 
It is well known that both the trace map and  determinant map 
are  completely positive; see, e.g., \cite[p. 221, p. 237]{Zhang11}. 
On the other hand, a linear map  $\Phi $  is said to be $m$-copositive if 
for $[A_{i,j}]_{i,j=1}^m\in \mathbb{M}_m(\mathbb{M}_n)$, 
\begin{equation}  \label{eq2}
[A_{i,j}]_{i,j=1}^m \ge 0 \Rightarrow [\Phi (A_{j,i})]_{i,j=1}^m\ge 0,  
\end{equation}
and $\Phi$ is said to be {\it completely copositive} 
if (\ref{eq2}) holds for any positive integer $m\ge 1$. 
Furthermore, 
$\Phi$ is called {\it a completely PPT map} if it is completely positive and completely copositive. 
A comprehensive survey of the standard results on completely positive maps can be found in 
\cite[Chapter 3]{Bh07} or \cite{Paulsen02}. 

We need  the following lemma, 
which is the main result in \cite{Lin14}; see, e.g., \cite{Li20laa}.

\begin{lemma} (see \cite{Lin14}) \label{lem21}
The map  $\Phi(X)=(\tr X)I +X$ is a completely PPT map. 
\end{lemma}

In the proof of the next proposition, 
we only employ the fact that $\Psi (X)=(\tr X)I+X$ is $2$-copositive. 
Proposition \ref{thm22}, first proved by the authors \cite{HuangLi20} recently, 
which is a complement of Ando's result (\ref{eqando}) and 
play a vital role in our derivation of Theorem \ref{thm34}. 
We here provide an alternative proof for convenience of readers. 
Our proof is slightly more transparent than the original proof in \cite{HuangLi20}. 

  \begin{proposition}\label{thm22} 
Let $A=[A_{i,j}]_{i,j=1}^m\in \mathbb{M}_m(\mathbb{M}_n)$ be positive semidefinite. Then
  	\begin{equation}\label{eqmain} 
(\tr A)I_m\otimes I_n - (\tr_2 A) \otimes I_n   
\ge  A -I_m\otimes (\tr_1 A) .
  	\end{equation}  
 \end{proposition}

  	\begin{proof} 
The proof is by induction on $m$. When $m=1$, there is nothing to prove. 
We now prove the base case   $m=2$. In this case, the required inequality is  
\begin{align*}  
& \begin{bmatrix}  (\tr A )I_n& 0\\ 0& (\tr A)I_n \end{bmatrix} - 
\begin{bmatrix} (\tr A_{1,1})I_n& (\tr A_{1,2})I_n\\ (\tr A_{2,1})I_n & (\tr A_{2,2})I_n \end{bmatrix} 
  \\
 &\quad \ge  \begin{bmatrix} A_{1,1}& A_{1,2}\\ A_{2,1} & A_{2,2} \end{bmatrix} - 
\begin{bmatrix} A_{1,1}+A_{2,2}& 0\\ 	0& A_{1,1}+A_{2,2}\end{bmatrix}, 
 \end{align*}
 or equivalently (note that $\tr A=\tr A_{1,1} +\tr A_{2,2}$),	
 \begin{eqnarray}\label{proofe1}
 H:=\begin{bmatrix}
  		(\tr A_{2,2})I_n+ A_{2,2}&   -A_{1,2}-(\tr A_{1,2})I_n\\  
-A_{2,1}-(\tr A_{2,1})I_n& 	(\tr A_{1,1})I_n+ A_{1,1}
  	\end{bmatrix}\ge 0.  \end{eqnarray}
  
  		By   Lemma \ref{lem21}, we have
  			\begin{eqnarray*} \begin{bmatrix}
  					(\tr A_{1,1})I_n+ A_{1,1}& (\tr A_{2,1})I_n+ A_{2,1}\\ 	
(\tr A_{1,2})I_n+ A_{1,2}& 	(\tr A_{2,2})I_n+ A_{2,2}
  				\end{bmatrix}\ge 0,  \end{eqnarray*} and so 
  				\begin{eqnarray*} H=\begin{bmatrix}0 & -I_n\\  I_n& 0\end{bmatrix} 
\begin{bmatrix} (\tr A_{1,1})I_n+ A_{1,1}& (\tr A_{2,1})I_n+ A_{2,1}\\ 
(\tr A_{1,2})I_n+ A_{1,2}& (\tr A_{2,2})I_n+ A_{2,2} \end{bmatrix} 
\begin{bmatrix}0 &  I_n\\  -I_n& 0\end{bmatrix} \ge 0,  \end{eqnarray*} 
which confirms the desired  (\ref{proofe1}).

  		 Suppose the result (\ref{eqmain}) is true for $m=k-1>1$, 
and then we consider the case $m=k$, 
  		\begin{align*}  
\Gamma &:=(\tr A)I_k\otimes I_n+ I_k \otimes(\tr_1 A)-A-(\tr_2 A) \otimes I_n	\\
&= \left(\tr \sum_{i=1}^{k}A_{i,i}\right)I_k\otimes I_n
+I_k\otimes \left(\sum_{j=1}^k A_{j,j}\right)-A-([\tr A_{i,j}]_{i,j=1}^k) \otimes I_n   		\\
&= \begin{bmatrix}  \sum_{i=1}^{k-1}(\tr A_{i,i})I_n &  &&  \\ &\ddots  & \\& &   \sum_{i=1}^{k-1}(\tr A_{i,i})I_n& \\ & & &  0 \end{bmatrix} \\
&\quad  + \begin{bmatrix} (\tr A_{k,k})I_n &  &&  \\ &\ddots  & \\& & (\tr A_{k,k})I_n & \\ & & &   \sum_{i=1}^k(\tr A_{i,i})I_n \end{bmatrix} 	\\ 
&\quad   + \begin{bmatrix}  \sum_{i=1}^{k-1}  A_{i,i} &  &&  \\ &\ddots  & \\& &   \sum_{i=1}^{k-1}  A_{i,i}& \\ & & &  0 \end{bmatrix}+ \begin{bmatrix}   A_{k,k} &  &&  \\ &\ddots  & \\& &   A_{k,k} & \\ & & &   \sum_{i=1}^k  A_{i,i} \end{bmatrix}	\\ 
& \quad  - \begin{bmatrix}   A_{1,1} & \cdots  & A_{1, k-1}& 0 \\ \vdots & & \vdots  & \vdots \\ A_{k-1,1}& \cdots &  A_{k-1, k-1}  & 0 \\ 0 & \cdots &  0&  0 \end{bmatrix}-\begin{bmatrix}  0 & \cdots  & 0&  A_{1, k} \\ \vdots & & \vdots  & \vdots \\0& \cdots &  0  & A_{k-1,k}  \\ A_{k,1} & \cdots &  A_{k,k-1} &  A_{k,k} \end{bmatrix}   	\\
&\quad   - \begin{bmatrix}  (\tr A_{1,1})I_n & \cdots  & (\tr A_{1, k-1})I_n& 0 \\ 
 \vdots & & \vdots  & \vdots \\ (\tr A_{k-1,1})I_n& \cdots &  (\tr A_{k-1, k-1})I_n  & 0 \\ 0 & \cdots &  0&  0 \end{bmatrix} \\
&\quad  -\begin{bmatrix}  0 & \cdots  & 0& (\tr A_{1, k})I_n \\ 
  \vdots & & \vdots  & \vdots \\0& \cdots &  0  & (\tr A_{k-1,k})I_n  \\ 
  (\tr A_{k,1})I_n & \cdots &  (\tr A_{k,k-1})I_n &  (\tr A_{k,k})I_n \end{bmatrix}.
\end{align*} 	
  		After some rearrangement, we have
\[\Gamma =\Gamma_1 +\Gamma_2, \]
  		where 	
\begin{align*}  
\Gamma_1&:=\begin{bmatrix}  \sum_{i=1}^{k-1}(\tr A_{i,i})I_n &  &&  \\ &\ddots  & \\& &   \sum_{i=1}^{k-1}(\tr A_{i,i})I_n& \\ & & &  0 \end{bmatrix}
  		  +  \begin{bmatrix}  \sum_{i=1}^{k-1}  A_{i,i} &  &&  \\ &\ddots  & \\& &   \sum_{i=1}^{k-1}  A_{i,i}& \\ & & &  0 \end{bmatrix} \\ 
& \phantom{:}- \begin{bmatrix}   A_{1,1} & \cdots  & A_{1, k-1}& 0 \\ \vdots & & \vdots  & \vdots \\ A_{k-1,1}& \cdots &  A_{k-1, k-1}  & 0 \\ 0 & \cdots &  0&  0 \end{bmatrix} - \begin{bmatrix}  (\tr A_{1,1})I_n & \cdots  & (\tr A_{1, k-1})I_n& 0 \\ \vdots & & \vdots  & \vdots \\ (\tr A_{k-1,1})I_n& \cdots &  (\tr A_{k-1, k-1})I_n  & 0 \\ 0 & \cdots &  0&  0 \end{bmatrix},
  		\end{align*} 
and	
\begin{align*}  
\Gamma_2&:= \begin{bmatrix} (\tr A_{k,k})I_n &  &&  \\ &\ddots  & \\& & (\tr A_{k,k})I_n & \\ & & &   \sum_{i=1}^k(\tr A_{i,i})I_n \end{bmatrix}
  			+ \begin{bmatrix}   A_{k,k} &  &&  \\ &\ddots  & \\& &   A_{k,k} & \\ & & &   \sum_{i=1}^k  A_{i,i} \end{bmatrix}  \\
& \phantom{:} - \begin{bmatrix}  0 & \cdots  & 0&  A_{1, k} \\ \vdots & & \vdots  & \vdots \\0& \cdots &  0  & A_{k-1,k}  \\ A_{k,1} & \cdots &  A_{k,k-1} &  A_{k,k} \end{bmatrix}	 - 
\begin{bmatrix}  0 & \cdots  & 0& (\tr A_{1, k})I_n \\ 
 \vdots & & \vdots  & \vdots \\ 
0& \cdots &  0  & (\tr A_{k-1,k})I_n  \\ 
(\tr A_{k,1})I_n & \cdots &  (\tr A_{k,k-1})I_n &  (\tr A_{k,k})I_n \end{bmatrix} 		\\
&\phantom{:}= \begin{bmatrix}  (\tr A_{k,k})I_n+A_{k,k}  &   &  & -A_{1,k}-(\tr A_{1, k})I_n \\     
   & \ddots  &  & \vdots  \\  &  &  (\tr A_{k,k})I_n+A_{k,k}   &  -A_{k-1,k}-(\tr A_{k-1,k})I_n  \\ 
  -A_{k,1}-(\tr A_{k,1})I_n & \cdots &  -A_{k,k-1}-(\tr A_{k,k-1})I_n &  
   \sum_{i=1}^{k-1}\big((\tr A_{i,i})I_n+A_{i,i}\big)  \end{bmatrix}.
\end{align*} 	
  		Now by induction hypothesis, we get that $\Gamma_1$ is positive semidefinite. 
It remains to show that   $\Gamma_2$ is positive semidefinite. 
  		
Observing that  $\Gamma_2$ can be written as a sum of $k-1$  
 matrices, in which each summand is $*$-congruent to 
$$H_i:=\begin{bmatrix}  (\tr A_{k,k})I_n+A_{k,k}  &    -A_{i,k}-(\tr A_{i, k}))I_n \\   
 - A_{k,i}-(\tr A_{k,i})I_n &  (\tr A_{i,i})I_n+A_{i,i} \end{bmatrix}, \quad 
  		 i=1, 2,\ldots, k-1.$$  
Just like the proof of the base case, 
we infer from Lemma \ref{lem21} that $H_i\ge 0$ for all $i=1,2, \ldots, k-1$. 
Therefore, $\Gamma_2 \ge 0$, thus the proof of induction step is complete.  
  	   \end{proof}

Before showing our proof of Theorem \ref{thm34}, 
we need one more lemma for our purpose.

 \begin{lemma}(see \cite{Lin16})\label{lem31}  
Let $X, Y, W, Z\in \mathbb{M}_\ell$ be positive semidefinite.  
If $X+Y\ge W+Z$, $X\ge W$ and $X\ge Z$, then
 	\[ \label{lin} \det X+\det Y\ge \det W+\det Z. \]
\end{lemma}

We remark that  Lemma \ref{lem31} implies the determinantal inequality: 
\[ \det (A+B+C) +\det C \ge \det (A+C) +\det (B+C), \]
where $A,B$ and $C$ are positive semidefinite; 
see \cite{Lin14b} and \cite{Li20} for more details.  

\vspace{0.4cm}

We are now in a position to present the proof of Theorem \ref{thm34}. 

\vspace{0.4cm}

\noindent 
{\bf Proof of Theorem \ref{thm34}}~
In view of (\ref{eqlin}), it suffices to show 
\begin{equation} \label{eq111}
 (\tr A)^{mn} +\det (\tr_1 A)^m \ge \det A +\det (\tr_2 A)^n. 
\end{equation}
Let $X=(\tr A)I_m\otimes I_n, Y=I_m\otimes (\tr_1 A),W=A,Z=(\tr_2 A)\otimes I_n$, 
respectively. It is easy to see that 
\[ (\tr A)I_m =\sum_{i=1}^m (\tr A_{i,i})I_m =
\bigl( \tr (\tr_2 A)\bigr)I_m \ge \tr_2 A, \]  
which implies that $X\ge Z\ge 0$, and clearly $X\ge W\ge 0$. 
Moreover, by Proposition \ref{thm22}, 
$X+Y\ge W+Z$. That is, all conditions in Lemma \ref{lem31} are satisfied. Therefore, 
\begin{align*}
(\tr A)^{mn} + \det \bigl( I_m\otimes (\tr_1 A) \bigr)
\ge  \det A  +\det \bigl( (\tr_2 A) \otimes I_n \bigr).
\end{align*}
Since $\det (X\otimes Y) =(\det X)^n(\det Y)^m$  
for every $X\in \mathbb{M}_m$ and $Y\in \mathbb{M}_n$, 
this completes the proof.

\vspace{0.4cm}

Using the same idea in previous proof and 
combining \cite[Proposition 2.3]{HuangLi20}, 
one could also get the following Proposition \ref{prop35}. 
We  leave the details for the interested reader.

\begin{proposition} \label{prop35}
 Let $A\in \mathbb{M}_m(\mathbb{M}_n)$ be PPT. Then
\[ 
	(\tr A)^{mn} + \det(\tr_2 A)^n \ge 
   \det A + \det(\tr_1 A)^m .
\]
\end{proposition}

At the end of the paper, we extend the determinantal inequality (\ref{eq111}) to 
a larger class of matrices whose numerical ranges are contained in a sector. 
The same extension of (\ref{eqlin}) can be found in \cite{YLC19}. 
Before showing our extension, we first introduce some standard notations.

\vspace{0.3cm}

For $A\in \mathbb{M}_n$, the Cartesian (Toeptliz) decomposition 
$A=\Re A+i\Im A$, where $\Re A=\frac{1}{2}(A+A^*)$ and $\Im A=\frac{1}{2i}(A-A^*)$. 
Let $|A|$ denote the positive square root of $A^*A$, 
i.e., $|A|=(A^*A)^{1/2}$.  We denote the $i$-th largest singular value of $A$ by 
$s_i(A)$, then $s_i(A)=\lambda_i(|A|)$, 
the $i$-th largest eigenvalue of $|A|$. 
Recall that the numerical range of $A\in \mathbb{M}_n$ is defined by 
\[ W(A)=\{x^*Ax : x\in \mathbb{C}^*,x^*x=1\}. \]
For $\alpha \in [0,\frac{\pi}{2})$, let $S_{\alpha}$ be the sector on the complex plane given by 
\[ S_{\alpha}=\{z\in \mathbb{C}: \Re z>0,|\Im z|\le (\Re z)\tan \alpha \} 
=\{re^{i\theta } : r>0,|\theta |\le \alpha \}. \]
Obviously, if $W(A)\subseteq S_{\alpha}$ for $\alpha \in [0,\frac{\pi}{2})$, 
then $\Re (A)$ is positive definite and if $W(A)\subseteq S_0$, 
then $A$ is positive definite. 
Such class of matrices whose numerical ranges are contained in a sector 
is called the {\it sector matrices class}. 
Clearly, the concept of sector matrices is an extension of that of positive definite matrices. 
Over the past years, 
various studies on sector matrices have been obtained in the literature; 
see, e.g., \cite{Choi19, Jiang19, Kua17, Lin15, YLC19, Zhang15}.

First, we list two lemmas which are useful to establish our extension (Theorem \ref{thm36}). 

\begin{lemma} (see \cite{Lin15}) \label{lem34}
Let $0\le \alpha <\frac{\pi}{2}$ and 
$A\in \mathbb{M}_n$ with $W(A)\subseteq S_{\alpha}$. Then 
\[ |\det A| \le (\sec \alpha)^n \det (\Re A). \]
\end{lemma}

\begin{lemma} (see {\cite[p. 510]{HJ13}}) \label{lem35}
Let $A$ be an $n$-square complex matrix. Then 
\[ \lambda_i(\Re A) \le s_i(A),\quad i=1,2,\ldots ,n. \]
Moreover, if $X$ has positive definite real part, then 
\[ \det \Re A + |\det \Im A| \le |\det A|. \]
\end{lemma}

Now, we  provide the extension of (\ref{eq111}). 

\begin{theorem} \label{thm36}
Let $H\in \mathbb{M}_m(\mathbb{M}_n)$ be such that $W(H)\subseteq S_{\alpha}$. Then 
\[ (\tr |A|)^{mn} + |\det (\tr_1 A)|^m \ge 
(\cos \alpha )^{mn} \det |A| + (\cos \alpha)^{mn} |\det (\tr_2 A)|^n. \]
\end{theorem}

\begin{proof}
By Lemma \ref{lem35}, we have 
\begin{equation}\label{eq12} 
\tr |A| =\sum_{i=1}^{mn} s_i(A) \ge \sum_{i=1}^{mn} \lambda_i(\Re A) =
\tr (\Re A) \ge 0. 
\end{equation}
Since $W(A)\subseteq S_{\alpha}$, 
it is noteworthy that $W(\tr_1 A) \subseteq S_{\alpha}$ and 
$W(\tr_2 A) \subseteq S_{\alpha}$; see, e.g.,  \cite{Kua17}. 
Observe that $\Re (\tr_1 A) = \tr_1( \Re A)$ and $\Re( \tr_2  A) =  \tr_2( \Re A)$. 
Therefore,  
\begin{align*}
(\tr |A|)^{mn} + |\det (\tr_1 A)|^m 
&\ge (\tr \Re A)^{mn} + \det \bigl( \Re (\tr_1 A) \bigr)^m \\
& = (\tr \Re A)^{mn} + \det \bigl(  \tr_1( \Re A) \bigr)^m \\
& \ge \det (\Re A) + \det \bigl(  \tr_2( \Re A) \bigr)^n \\
&= \det (\Re A) + \det \bigl(  \Re( \tr_2  A) \bigr)^n \\ 
&\ge (\cos \alpha)^{mn} |\det A| + (\cos \alpha)^{mn} |\det (\tr_2 A)^n |, 
\end{align*}
where the first inequality follows from (\ref{eq12}) and Lemma \ref{lem35},  
the second one follows by applying (\ref{eq111}) to $\Re A$,
the last one is by Lemma \ref{lem34}. 
\end{proof}

\section*{Acknowledgments}
The author would like to thank Dr. Minghua Lin  
for bringing the question to his attention and for naming the title of the manuscript, 
which can be regarded as a continuation and development of his result \cite{Lin16}.  
All authors are grateful for valuable comments from the referee, 
which considerably improve the presentation of our manuscript.
This work was supported by  NSFC (Grant Nos. 11671402, 11871479),  
Hunan Provincial Natural Science Foundation (Grant Nos. 2016JJ2138, 2018JJ2479) 
and  Mathematics and Interdisciplinary Sciences Project of CSU.

\end{document}